\documentclass[10pt,a4paper]{article}
\usepackage[T1]{fontenc}
\usepackage[utf8]{inputenc}
\usepackage{authblk}
\usepackage{amsmath,amssymb,amsthm,esint,bm}
\usepackage{mathrsfs}
\usepackage{bookmark}
\usepackage{amsmath}
\allowdisplaybreaks[3]

\newtheorem{definition}{Definition}[section]
\newtheorem{theorem}[definition]{Theorem}

\theoremstyle{remark}
\newtheorem{remark}[definition]{Remark}
\numberwithin{equation}{section}

\newcommand{\abs}[1]{\lvert#1\rvert}

\newcommand{\R}{\mathbb{R}}
\newcommand{\cL}{\mathcal{L}}

\newcommand{\fr}{\displaystyle\frac}
\newcommand{\jf}{\displaystyle\int}
\newcommand{\lt}{\left}
\newcommand{\rt}{\right}
\newcommand{\lm}{\lambda}

\newcommand{\Om}{\Omega}
\newcommand{\Sm}{\Sigma}
\newcommand{\mb}{\mbox}

\newcommand{\tm}{\times}
\newcommand{\ve}{\varepsilon}

\usepackage{extarrows}

\title{ A Liouville Theorem and Radial Symmetry for dual fractional  parabolic equations}

\author[a,c]{Yahong Guo}
\author[b,c]{Lingwei Ma}
\author[d]{Zhenqiu Zhang\thanks{Corresponding author.}}

\affil[a]{School of Mathematical Sciences, Nankai University, Tianjin, 300071, P.~R.~China}
\affil[b]{School of Mathematical Sciences, Tianjin Normal University, Tianjin, 300387, P.~R. ~China}
\affil[c]{Department of Mathematical Sciences, Yeshiva University, New York, NY, 10033, USA}
\affil[d]{School of Mathematical Sciences and LPMC, Nankai University, Tianjin, 300071, P.~R.~China}

\usepackage{hyperref}
\begin{document}
\maketitle
\footnotetext[1]{E-mail: guoyahong1995@outlook.com (Y. Guo), mlw1103@163.com (L. Ma), zqzhang@nankai.edu.cn (Z. Zhang) .}

\begin{abstract}

In this paper, we first study the  dual fractional parabolic equation
\begin{equation*}
\partial^\alpha_t u(x,t)+(-\Delta)^s u(x,t) =  f(u(x,t))\ \ \mbox{in}\ \ B_1(0)\times\R ,
\end{equation*}
subject to the vanishing exterior condition. We show that for each $t\in\R$, the positive bounded  solution $u(\cdot,t)$  must be radially symmetric and strictly decreasing about the origin in the unit ball in $\R^n$. 
 To overcome the challenges caused by the dual non-locality of the operator $\partial^\alpha_t+(-\Delta)^s$, some novel techniques were introduced.

 Then we establish  the Liouville theorem for the homogeneous equation in the whole space
\begin{equation*}\label{B}
\partial^\alpha_t u(x,t)+(-\Delta)^s u(x,t) = 0\ \ \mbox{in}\ \ \R^n\times\R .
\end{equation*}
 We first prove a   maximum principle in unbounded domains for  anti-symmetric functions to  deduce that $u(x,t)$ must be constant with respect to  $x.$
  Then it suffices for us  to  establish the Liouville theorem for  the Marchaud fractional equation
 \begin{equation*}
\partial^\alpha_t u(t) = 0\ \ \mbox{in}\ \ \R .
\end{equation*}
To circumvent the difficulties arising from the nonlocal and one-sided nature of the operator $\partial_t^\alpha$, we  bring in  some new ideas and  simpler approaches. Instead of disturbing the anti-symmetric function, we employ a perturbation technique directly on the solution $u(t)$ itself.   This method provides a more concise and intuitive route to establish the Liouville theorem for one-sided operators $\partial_t^\alpha$, 
including even  more general Marchaud time derivatives.


\end{abstract}

\section{Introduction}
\
\indent  The primary objective of this paper is to investigate the qualitative properties of solutions to  dual nonlocal parabolic equations associated with the operator $\partial_t^\alpha+(-\Delta)^s$. More precisely, we first investigate the radial symmetry and monotonicity of solutions for the following equation in the unit ball
\begin{equation}\label{A}
\left\{
\begin{array}{ll}
\partial^\alpha_t u(x,t)+(-\Delta)^s u(x,t) =  f(u(x,t))\ \ &\mbox{in}\ \ B_1(0)\times\R ,\\
u(x,t)\equiv 0 \ \ &\mbox{in}\ \ B_1^c(0)\times\R.
\end{array}
\right.
\end{equation}
Then we  establish the Liouville theorem  for the homogeneous equation in the whole space
\begin{equation}\label{B}
\partial^\alpha_t u(x,t)+(-\Delta)^s u(x,t) = 0\ \ \mbox{in}\ \ \R^n\times\R .
\end{equation}

 The one-sided nonlocal time derivative $\partial_t^\alpha$ considered here is known as the Marchaud fractional derivative  of order $\alpha$, defined as
\begin{equation}\label{1.00}
\partial^\alpha_t u(x,t)=C_\alpha\jf_{-\infty}^t\fr{u(x,t)-u(x,\tau)}{(t-\tau)^{1+\alpha}}d\tau,
\end{equation}
with $0<\alpha<1, C_\alpha=\fr{\alpha}{\Gamma({1-\alpha})}$ and $\Gamma$ represents the Gamma function.
 From the definition, such fractional time derivative  depends on the values of  function  from the past, sometime also denoted as $(D_{\rm{left}})^\alpha$. 
The spatial nonlocal elliptic pseudo-differential operator, the fractional Laplacian $(-\Delta)^s$
is defined as
 \begin{equation}\label{1.0}
(-\Delta)^su(x,t)= C_{n,s}P.V.\jf_{\R^{n}}\fr{u(x,t)-u(y,t)}{\abs{x-y}^{n+2s}}dy.
\end{equation}
where $0 < s < 1$, $C_{n,s} := \frac{4^s \Gamma\left(\frac{n+2s}{2}\right)}{\pi^{n/2} \left|\Gamma(-s)\right|}$ is a normalization positive constant and $P.V.$ stands for the Cauchy principal
value.
In order to  guarantees that the singular integral in  \eqref{1.00} and \eqref{1.0}  are
well defined, we assume that    $$u(x,t)\in\left(\cL_{2s}\cap C^{1,1}_{loc}(\R^n)\right)\times\left(C^1(\R)\cap \cL^{-}_\alpha(\R)\right),$$  Here,  the slowly increasing function spaces $\cL_{2s}$ and $\cL^{-}_\alpha(\R)$ are defined respectively by

 $$\cL_{2s}:=\lt\{v\in L^1_{loc}(\R^n) \ | \int_{\mathbb{R}^n} \frac{|v(x)|}{1+|x|^{n+2s}} dx < +\infty\rt\}$$
 and
 $$\cL^{-}_\alpha(\R):=\lt\{v\in L^1_{loc}(\R)\ | \int_{-\infty}^{t} \frac{|v(\tau)|}{1+|\tau|^{1+\alpha}} d\tau < +\infty \text{ for\ each}\  t\in\mathbb{R}\rt\}.$$

A typical application of equation in \eqref{A} is in modeling  continuous time random walks \cite{33}, which generalizes Brownian random walks. This fractional kinetic equation introduces nonlocality in time, leading to history dependence due to unusually large waiting times, and nonlocality in space, accounting for unusually large jumps connecting distant regions, such as L$\acute{e}$vy flights. In applications within financial field, it can also be used to model the waiting time between
transactions is correlated with the ensuring price jump (cf. \cite{37}). Another model  is presented in \cite{dCN} to simulate transport of tracer particles in plasma, where the function $u$ is the probability density
function for tracer particles which represents the probability of finding a particle
at time $t$ and position $x$, the right hand side $f$ is a source term. In this case, the  nonlocal space operator $(-\Delta)^s$  accounts for avalanche-like transport that can occur, while the Marchaud time derivative  $\partial^\alpha_t$ accounts
for the trapping of the trace particles in turbulent eddies. It is worth mentioning that the nonlocal operator \(\partial^\alpha_t + (-\Delta)^s\) in problem \eqref{A} can be reduced to the local heat operator \(\partial_t - \Delta\) as \(\alpha \to 1\) and \(s \to 1\).

The method of moving planes, initially introduced by Alexandroff in \cite{28} and simplified by Berestycki and Nirenberg \cite{HB3}, is a widely used technique for studying the monotonicity of solutions to local elliptic and parabolic equations. However, this approach can not be applied directly to psuedo-differential equations involving the fractional Laplacian
due to its nonlocality. To circumvent this difficulty, Cafferelli and Silvestre \cite{14} introduced an extension method that can turn a non-local equation to a local one in higher dimensions. Thereby the traditional method
of moving planes  designed for local equations  can be applied for the extended problem to establish the well-posedness of solutions,  and   a series of interesting results have been obtained in \cite{Chang,11,  13, 18, Sun1,20,Sun2, LC2,31} and the references therein.
However, this method is exclusively applicable to  equations involving the fractional Laplacian and sometimes additional restrictions may need to be imposed on the problems, while it  will not be necessary in dealing with the fractional equations directly. To remove these restrictions, Chen, Li, and Li \cite{11} introduced a direct method of moving planes  nearly ten years later. This method  significantly simplify the proof process and has been widely applied to establish the symmetry, monotonicity, non-existence of solutions
for various elliptic equations and systems involving the fractional Laplacian, the fully nonlinear
nonlocal operators, the fractional p-Laplacians as well as the higher order fractional operators, we
refer to \cite{ CL,CLL,CW,DQ,LC3,34} and the references therein.  Recently, this method has also been gradually made use of  studing the geometric behavior of solutions for fractional parabolic equations  with the general local time derivative $\partial _t u(x,t)$. (cf. \cite{10,CM,JW,WuC}
and the references therein).
In particular,  the authors of  \cite{10}  established symmetry and monotonicity of positive solutions on a
unit ball for the classical parabolic problem
\begin{equation}\label{Local}
\left\{
\begin{array}{ll}
\partial_t u(x,t)+(-\Delta)^s u(x,t) =  f(u(x,t))\ \ &\mbox{in}\ \ B_1(0)\times\R ,\\
u(x,t)\equiv 0 \ \ &\mbox{in}\ \ B_1^c(0)\times\R.
\end{array}
\right.
\end{equation}
However, so far as we know, there is still a lack of research on the geometric properties of solutions to nonlocal parabolic equations \eqref{A} with  the Marchaud fractional time derivative $\partial _t^\alpha u(x,t)$ and the fractional Laplacian$(-\Delta)^s$.
Recently, Guo, Ma and Zhang \cite{24} employed a suitable sliding method, first introduced by Berestycki and Nirenberg \cite{HB3}, to demenstrate the generalized version of Gibbons' conjecture in the setting of
the dual  nonlocal parabolic equation
 \begin{equation*}
\partial^\alpha_t u(x,t)+\cL u(x,t) =  f(t,u(x,t))\ \ \mbox{in}\ \ \R^n\times\R.
\end{equation*}
Here the spatial nonlocal elliptic operators of integro-differential type is defined as
 \begin{equation}\label{D}
\cL u(x,t)= P.V.\jf_{\R^{n}}\left[u(x,t)-u(y,t)\right]\cdot K(x,y)dy.
\end{equation}
  Chen and Ma \cite{cm} carried out  a suitable direct method of moving planes to obtain the monotonicity of positive solutions for the following problem
\begin{equation*}
\left\{
\begin{array}{ll}
\partial^\alpha_t u(x,t)+(-\Delta)^s u(x,t) =  f(u(x,t))\ \ &\mbox{in}\ \mathbb{R}^n_+ \times \mathbb{R}, ,\\
u(x,t)\equiv 0 \ \ &\mbox{in}\ \ (\mathbb{R}^n \setminus \mathbb{R}^n_+) \times \mathbb{R}.
\end{array}
\right.\end{equation*}
Therefore, our first main interest here is to apply a direct method of moving planes to establish the radial symmetry and monotonicity of solutions to problem \eqref{A} in the unit ball.

Our second main objective is to establish the Liouville theorem of equation \eqref{B}. The classical  Liouville theorem states that any bounded harmonic function defined in the entire space $\mathbb{R}^n$  must be identically constant. This theorem plays a crucial role in deriving a priori estimates and establishing the qualitative properties of solutions, including their existence, nonexistence, and uniqueness. As a result, it has been extensively studied in the analysis of partial differential equations and this area of study has been further extended to various types of  elliptic and fractional elliptic  equations, even to $k$-Hessian equations using diverse methods, including Harnack inequalities, blow-up and compactness arguments, as well as Fourier analysis (cf. \cite{Chang1,BK,18,DQ,F,M,SZ} and the references therein).

In the context of nonlocal homogeneous parabolic equation \eqref{B}, when restricted the domain of $t$ to $(-\infty, 0]$,
 Widder \cite{W} proved that all bounded solutions $u(x, t)$ must be constant in case of $\alpha=1,s=1$; while for $\alpha=1,0<s<1$,
Serra \cite{S} showed that the solutions with some growth condition  is a constant.
In recent times, Ma, Guo and Zhang \cite{MGZ} demonstrated that the bounded entire solutions of the homogeneous
master equation
\begin{equation}\label{C}(\partial_t -\Delta)^s u(x, t) = 0 \mb{~in~} \mathbb{R}^n \times \mathbb{R},\end{equation}
must be constant.
Here the fully fractional heat operator $(\partial_t -\Delta)^s$ was first proposed by Riesz \cite{R}, and it can be defined pointwise using the following singular integral:
$$(\partial_t -\Delta)^s u(x, t) := C_{n,s} \jf_{-\infty}^{t} \jf_{\mathbb{R}^n} \frac{u(x, t) - u(y, \tau)}{(t - \tau)^{n+2s}} e^{-\frac{|x - y|^2 }{ 4(t - \tau)}} dy d\tau,$$
where $0 < s < 1$,
$C_{n,s} = \frac{1}{(4\pi)^{n/2} |\Gamma(-s)|}$.
It is essential to emphasize that in \cite{MGZ}, we first established the maximum principles for  operators  $(\partial_t -\Delta)^s$ to conclude that any bounded solution $u(x, t)$ must be  constant with respect to the spatial variable $x$. i.e. $u(x,t)=u(t).$ This will simplify  equation \eqref{C} to a one-sided one-dimensional fractional equation
\begin{equation}\label{time}\partial^\alpha_t u(t)=0 ~ \mbox{in}\ \ \R.\end{equation}
Then we  obtained that the bounded solution $u(t)$ must be  constant with respect to $t$ by employing  the method of Fourier analysis which is  applicable to more general distributions beyond bounded functions. While this method  does not fully capture the one-sided nature of the operator  $\partial_t^\alpha.$
Taking inspiration from these findings, our second main objective is to develop an alternative and more straightforward method  to generalize the Liouville theorem to the dual fractional parabolic  operator $\partial_t^\alpha+(-\Delta)^s$ in the whole space.

Now we explain the novelty and challenges of our approach in
deriving the radial symmetry of solutions for problem \eqref{A} in the unit ball and the Liouville theorem for equation \eqref{B} in the whole space
by analysing the characteristics of the one-sided fractional time operator $\partial_t^\alpha$ and the (double-sided) fractional Laplacian $(-\Delta)^s$.

  In comparison with \cite{10} for the operator $\partial_t+(-\Delta)^s$ and \cite{MGZ} for the operator $(\partial_t -\Delta)^s$, a notable difference in this paper is that all perturbations are novel and constructed from  different scaling and shifting of smooth cut-off functions $\eta_k$ to match the dual fractional parabolic operators $\partial_t^\alpha+(-\Delta)^s$. Then
by applying the $\mathbf{Translation~ and ~Rescaling ~Invariance} $
\begin{equation}\label{TS}\cL\left[u\lt(\fr{x-\bar{x}}{r}\rt)\right] = \frac{1}{r^{\beta}}\cL u\left(\fr{x-\bar{x}}{r}\right),\end{equation}
to the specific operators $\cL=(-\Delta)^s$ with $\beta=2s$ and   $\cL=\partial_t^\alpha$ with $\beta=\alpha$, we  derive $$\cL\eta_k\lesssim \fr{1}{r^\beta},$$ which is a key estimate in proving the maximum principle for anti-symmetry functions with respect to $x$ as well as the Liouville theorem for the   Marchaud fractional operator $\partial_t^\alpha.$ Utilizing these essential tools, we can develop the direct  method of moving planes and obtain the Liouville theorem for the dual fractional operator $\partial_t^\alpha+(-\Delta)^s$ .

From one aspect, we  point out the distinction between the local time derivatives $\partial _t$ and the nonlocal operator $\partial_t^\alpha$   in the process of establishing the radial symmetry of solutions for equation \eqref{A} through the direct method of moving planes combined with the limiting argument.  Differing from 
traditional approaches employed for classical parabolic equations \eqref{Local}(cf. \cite{10}), we repeatedly  use the following two key $\mathbf{observations}$ arising from the nonlocal and one-sided nature of the one-dimensional fractional time operator $\partial_t^\alpha$.

$\mathbf{Observation~A.}$ If $u(\bar{t})=\min\limits_{t\in\R}u(t) ~(\mb{or~}\max\limits_{t\in\R}u(t))$, then $\partial_t^\alpha u(\bar{t})\leq 0~(\mb{or~}\geq 0).$

$\mathbf{Observation~B.}$ Assume that $u(\bar{t})=\min\limits_{t\in\R}u(t) ~(\mb{or~}\max\limits_{t\in\R}u(t))$. Then    $\partial_t^\alpha u(\bar{t})= 0$ if and only if $$u(t)\equiv u(\bar{t}) \mb{~in~}t<\bar{t}.$$

From another standpoint, we emphasize the different challenges between the one-sided operator $\partial_t^\alpha$ and the fractional
Laplacian $(-\Delta)^s$
in the process of deriving the Liouville theorem for homogeneous equation \eqref{B}. It is well-known that  the (double-sided) fractional Laplacian $(-\Delta)^s$ satisfies the
$\mathbf{Reflection ~Invariance}$ (or chain rule)
  \begin{equation}\label{R}(-\Delta)^s\left[u(x^\lm)\right] = (-\Delta)^s u\left(x^\lm\right), \end{equation}
 where $x^\lambda = (2\lambda - x_1, x')$ denote  the reflection of $x$ with respect to the hyperplane $x_1=\lm.$ 
However  this is no longer valid for the  fractional time derivative $\cL=\partial_t^\alpha$ due to its one-sided nature. Indeed, if we denote $(D_{\rm{left}})^\alpha:=\partial_t^\alpha$ and $t^\lm=2\lm-t$,  then instead of \eqref{R} we   obtain \begin{equation*}\label{R1}(D_{\rm{left}})^\alpha\left[u(t^\lm)\right] = (D_{\rm{right}})^\alpha u\left(t^\lm\right).\end{equation*}
Here $(D_{\rm{right}})^\alpha$ is also a fractional derivative that based on the values of the function  in the future, defined as \[(D_{\rm{right}})^\alpha u(t):=C_\alpha\jf_t^{+\infty}\fr{u(t)-u(\tau)}{(\tau-t)^{1+\alpha}}d\tau.\]

The property \eqref{R} plays a crucial role in establishing the   symmetry of solutions with respect to spatial planes and further deriving the Liouville theorem. Let us compare equation \eqref{B} with  the classical fractional  parabolic equation
\begin{equation}\label{R2}
\partial_tu(x,t)+(-\Delta)^su(x,t)=0 \mb{~in~} \R^n\tm\R.
\end{equation}
By establishing the maximum principle for anti-symmetric function $w(x,t)=u(x^\lm,t)-u(x,t),$ we  conclude that any bounded solution $u(\cdot, t)$ is symmetric with respect to any hyperplane in $\R^n$ for each fixed $t\in\R$, i.e.,
\[u(x,t)=u(t) \mb{~in~}\R^n\tm\R,\]
and hence $\partial_tu(t)=0.$ From this one can derive immediately  $u$, a bounded solution of equation \eqref{R2}, is a constant. 
However for the dual fractional parabolic equation \eqref{B},  it still need  to  further prove  the Liouville theorem for  Marchaud fractional equation \eqref{time}.  Due to the lack of reflection invariance \eqref{R} for one-sided operator  $\partial_t^\alpha$, one can not establish a maximum principle for the antisymmetric function \(w (t)= u(t^\lm) - u(t)\) in the same way as with double-sided operators like the fractional Laplacian. To circumvent this difficulty, 
in this paper, we  introduce some new ideas and simpler approaches. Inspired by the aforementioned $\mathbf{Observation~ }\mathbf{B} $ satisfied by the one-sided operator  itself, we directly begin  with the  definition of operator $\partial_t^\alpha$ and employ a perturbation technique on the solution $u(t)$ itself instead of on the anti-symmetric function $w(t)$.    It provides a more concise and intuitive method for establishing the Liouville theorem for one-sided operators $\partial_t^\alpha$. This is precisely a novel aspect of our work. In contrast to the Fourier analysis method used in our recent work \cite{MGZ}, this refined approach   highlights more directly the distinctions between one-sided and double-sided operators. 
Needless to say, focusing on  the nonlocal time operator, we work mainly with  $(D_{\rm{left}})^\alpha$. While all our results are equally valid for the right fractional time derivitive $(D_{\rm{right}})^\alpha$. In addition, it is  notable to emphasize that the proofs presented here for the
radial symmetry and monotonicity of solutions as well as the Liouville theorem  can be adapted to various nonlocal  equations involving  the spatial nonlocal elliptic operators $\cL$ as defined in \eqref{D} and the general
fractional time derivative (cf. \cite{1,2}) of the form
\begin{equation*}\label{GMD}
\int_{-\infty}^{t} [u(t)-u(s)]\mathcal{K}(t,s) ds.
\end{equation*}
Provided that the kernel $\mathcal{K}$ here and $K$ in \eqref{D}
possesses some radial decreasing property.

 Before presenting   the main results of this paper, we introduce the notation that will be used throughout the subsequent sections. Let $x_1$ be any given direction in $\mathbb{R}^n$,  $$T_\lambda = \{(x_1, x') \in \mathbb{R}^n \mid x_1 = \lambda, \lambda \in \mathbb{R}\}$$
be a moving planes perpendicular to the $x_1$-axis,
$$\Sigma_\lambda = \{x \in \mathbb{R}^n \mid x_1 < \lambda\}$$ be the region to the left of the hyperplane $T_\lambda$ in $\mathbb{R}^n$ and $$\Omega_\lambda = \Sigma_\lambda\cap B_1(0).$$
Furthermore, we denote the reflection of $x$ with respect to the hyperplane $T_\lambda$ as $$x^\lambda = (2\lambda - x_1, x_2, \ldots, x_n).$$
Let $u_\lambda(x, t) = u(x^\lambda, t)$, we define $$w_\lambda(x, t) = u_\lambda(x, t) - u(x, t).$$ It is evident that $w_\lambda(x, t)$ is an antisymmetric function of $x$ with respect to the hyperplane $T_\lambda$.

We are now ready to illustrate the main results of this paper.

\begin{theorem}\label{thm2}
Let $u(x,t)\in\left(C^{1,1}(B_1(0))\cap C(\overline{B_1(0)})\right)\times C^1(\R)$ be a  positive bounded solution of
\begin{equation}\label{1.11}
\left\{
\begin{array}{ll}
\partial^\alpha_t u(x,t)+(-\Delta)^s u(x,t) =  f(u(x,t))\ \ &\mbox{in}\ \ B_1(0)\times\R ,\\
u(x,t)\equiv 0 \ \ &\mbox{in}\ \ B_1^c(0)\times\R.
\end{array}
\right.
\end{equation}
Suppose that $f\in C^1([0,+\infty))$ satisfies $f(0)\geq 0$ and $f'(0)\leq 0.$
Then for each $t\in\R,$ $u(\cdot,t)$ is radially symmetric and strictly decreasing about the origin in $B_1(0)$.
\end{theorem}


The Theorem \ref{thm2} is proved by using  the direct method of moving plane  for dual fractional operators $\partial_t^\alpha+(-\Delta)^s$, which primarily relies on the following narrow region  principle 
 for anti-symmetric functions.

\begin{theorem}\label{thm1}
Let $\Omega$ be a bounded domain containing in the slab $\{x\in\Sigma_{\lambda}\mid\lambda-l<x_1<\lambda\}$.
Assume that $w(x,t)\in\left(\cL_{2s}\cap C^{1,1}_{loc}(\Omega)\right)\times\left(C^1(\R)\cap \cL^{-}_\alpha(\R)\right)$ is bounded from below in $\overline{\Omega}\times \R$ and for each $t\in \R,$ $w(\cdot,t)$ is lower semi-continuous up to the boundary $\partial \Omega$. Suppose
\begin{equation}\label{1.1}
\left\{
\begin{array}{ll}
    \partial^\alpha_t w(x,t)+(-\Delta)^s w(x,t)=c(x,t)w(x,t) ,~   &(x,t) \in  \Omega\times\mathbb{R}  , \\
  w(x,t)\geq 0 , ~ &(x,t)  \in (\Sigma_\lambda\backslash\Omega)\times\mathbb{R},\\
  w(x,t)=-w(x^\lambda,t) , ~ &(x,t)  \in \Sigma_\lambda\times\mathbb{R}.
\end{array}
\right.
\end{equation}
where the coefficient function $c(x,t)$ is bounded from above.\\
Then
\begin{equation}\label{1.2}
w(x,t)\geq 0~\mb{in}~\Sigma_\lambda\tm\R,
\end{equation}
for sufficiently small $l$. Furthermore, if $w(x,t)$ vanishes at some point $(x^0,t_0)\in\Omega\tm\R,$ then
\begin{equation}\label{1.3}
w(x,t)\equiv0~\mb{in}~\R^n\tm(-\infty,t_0].
\end{equation}

\end{theorem}

It is worth noting that in  theorem \ref{thm1}, $\Omega$ is  a bounded narrow domain within $\Sigma_\lm$ and $c(x,t)$ is just  bounded from above.  However for the whole unbounded region $\Sigma_\lm$ restricted to $w>0$, espectially when $c(x,t)$ is nonpositive, we will also have the second Maximum Principle for anti-symmetric functions with respect to $x$. This serves as   a  fundamental  tool in estabishing the Liouville theorem for the dual fractional operator $\partial_t^\alpha+(-\Delta)^s$.
\begin{theorem}\label{thm3}
Assume that $w(x,t)\in\left(\cL_{2s}\cap C^{1,1}_{loc}(\Sigma_\lm)\right)\times\left(C^1(\R)\cap \cL^{-}_\alpha(\R)\right)$ is bounded from above in $\Sigma_\lm\times \R$ 
and satisfies
\begin{equation}\label{2.1}
\left\{
\begin{array}{ll}
    \partial^\alpha_t w(x,t)+(-\Delta)^s w(x,t)\leq 0 ,~ &\mbox{in}\ \ \{(x,t)\in \Sigma_\lambda\times\R\ | \ w(x,t)>0\}\,, \\[0.05cm]
  w(x,t)=-w(x^\lambda,t) , ~ & \mbox{in}\ \ \ \Sigma_\lambda\times\mathbb{R}.
\end{array}
\right.
\end{equation}
Then
\begin{equation}\label{2.2}
w(x,t)\leq 0~\mb{in}~\Sigma_\lambda\tm\R.
\end{equation}
\end{theorem}
Since $w(x,t)=u(x^\lm,t)-u(x,t)$ is an anti-symmetric function with respect to  $x$,  Theorem \ref{thm3} only yields that a bounded entire solution $u(x,t)$ of homogeneous equation associated with the operator  $\partial_t^\alpha+(-\Delta)^s$ in the whole space $\R^n\tm\R$ must be constant with respect to the spatial variable $x$, i.e. $u(x,t)=u(t)$. To further show that it is also a constant with respect to the time variable $t$, it suffices for us  to establish a Liouville theorem involving a one-sided Marchaud fractional time operator $\partial^\alpha_t$ as the following.
\begin{theorem}\label{thm5}
Let $u(t)\in C^1(\R)$ be a  bounded solution of
\begin{equation}\label{anti1}
\partial^\alpha_t u(t) = 0\ \ \mbox{in}\ \ \R .
\end{equation}
Then it must be  constant.
\end{theorem}
As an immediate applications of the maximum principle in unbounded domains as stated in Theorem \ref{thm3} and the Liouville Theorem for the Marchaud  operator $\partial^\alpha_t$ in $\R$, Theorem \ref{thm5}, we derive  the second main result in this paper ---  Liouville Theorem for the dual fractional operator $\partial_t^\alpha+(-\Delta)^s$ in the whole space.
\begin{theorem}\label{thm4}
Let $u(x,t)\in C_{loc}^{1,1}(\R^n)\times C^1(\R) $ be a   bounded solution of
\begin{equation}\label{2.9}
\partial^\alpha_t u(x,t)+(-\Delta)^s u(x,t) = 0\ \ \mbox{in}\ \ \R^n\times\R .
\end{equation}
Then it must be  constant.
\end{theorem}
\begin{remark}The above theorem  can be
regarded as a generalization of the classical Liouville theorem for the  fractional elliptic and parabolic equation involving the Laplacian in the whole space, where the boundedness condition may not be optimal but is still reasonable. Relaxing this boundedness condition is the focus of our upcoming work.
\end{remark}
The remaining of this paper is organized as follows. In Sec.2,  we first  demonstrate two maximum
principle: the narrow domain principle (Theorem \ref{thm1}) and the maximum principle in unbounded
domains (Theorem \ref{thm3} ) applicable to the dual fractional operator $\partial^\alpha_t +(-\Delta)^s$.   Based on the narrow domain principle, we then carry out a direct method of
moving planes for the nonlocal operator $\partial^\alpha_t +(-\Delta)^s$ to prove the radial symmetry of solutions announced in Theorem \ref{thm2} in Sec.3.  Moving on to  Sec.4, we initially establish the Liouville theorem for the Marchaud  operator $\partial^\alpha_t$ (Theorem \ref{thm5}), and subsequently, in  combination with the maximum principle in unbounded
domains developed in Sec.2, we prove the Liouville Theorem for the dual fractional operator $\partial_t^\alpha+(-\Delta)^s$ as stated in Theorem \ref{thm4}. Throughout this
paper, we use $C$ to denote a general constant whose value may vary from line to line.

\section{Maximum Principles for Antisymmetric functions}
In this section, we will demonstrate various maximum principles for antisymmetric functions, including Theorem \ref{thm1} and Theorem \ref{thm3}. We will explain in the subsequent part how these principles play vital roles in carrying out a direct method of moving planes to establish the symmetry and monotonicity of solutions.
\subsection{Narrow region principle in bounded domains}
Our first key tool is  a narrow region principle for antisymmetric functions in bounded domains, which plays a crucial role in  deriving the radial symmetry and monotonicity of solutions for the dual fractional equation.
\begin{proof}[Proof of Theorem \ref{thm1}]
First we argue by contradiction to derive \eqref{1.2}. If not, since $\Omega$ is bounded, $w$ is bounded from below in $\Om\tm\R$ and $w(\cdot,t)$ is lower semi-continuous up to the boundary $\partial\Omega$ for each fixed $t\in\R$, there must exist $x(t)\in\Omega$ and  $m>0$ such that
\begin{equation}\label{1.4}\inf\limits_{(x,t)\in\Omega\tm\R}w(x,t)=\inf\limits_{t\in\R}w(x(t),t)=-m<0. \end{equation}
Then there exists a minimizing sequence $\{t_k\}\subset \R$ and a sequence $\{m_k\}\nearrow m$  such that
\[w(x(t_k),t_k)=-m_k\searrow-m~as~k\to\infty.\]
Since the infimum of $w$ with respect to $t$ may not be attained, we need to perturb $w$ with respect to $t$ such that the infimum $-m$  can be attained by the perturbed function. For this purpose, we introduce the following auxiliary function
\[v_k(x,t)=w(x,t)-\varepsilon_k\eta_k(t),\]
where $\varepsilon_k=m-m_k$ and $\eta_k(t)=\eta(t-t_k)$ with $\eta\in C_0^\infty(-1,1)$, $ 0\leq\eta\leq1$ satisfying
 \begin{equation*}\eta( t)  = \begin{cases} 1, & \abs{t}\leq \frac{1}{2}, \\ 0, &  \abs{t}\geq 1. \end{cases}\end{equation*}
Clearly $supp \eta_k\subset (-1+t_k,1+t_k)$ and $\eta_k(t_k)=1$. By \eqref{1.4} and the exterior condition in \eqref{1.1}, we have
\begin{equation*}
\left.\begin{array}{r@{\ \ }c@{\ \ }ll}
v_k(x(t_k),t_k)&=&-m\,, \\[0.15cm]
v_k(x,t)=w(x,t)&\geq& -m\ \mbox{in}\ \Omega\tm(\R \backslash (-1+t_k,1+t_k))\,, \\[0.15cm]
v_k(x,t)\geq-\varepsilon_k\eta_k(t) &>& -m\ \mbox{in}\ (\Sigma_\lambda\backslash\Omega)\times \R\,. \\[0.05cm]
\end{array}\right.
\end{equation*}

Since $w$ is lower semi-continuous on $\overline\Omega\times \R$, then $v_k$ must attains its minimum value which is at most  $-m$ at $\Omega\times (-1+t_k,1+t_k)$, that is,
\begin{equation}\label{1.5}
\exists\ \{(\bar{{x}}^k,\bar{t}_k)\}\subset\Omega\times (-1+t_k,1+t_k)\   \ s.t.\ \     -m-\varepsilon_k\leq v_k(\bar{{x}}^k,\bar{t}_k)=\inf\limits_{\Sigma_\lambda\times\R}v_k(x,t)\leq -m.
\end{equation}
Consequently,
\[-m\leq w(\bar{x}^k,\bar{t}_k)\leq-m_k<0.\]
Now applying \eqref{1.5}, the definition of $v_k$ and the anti-symmetry of $w$ in $x$, we derive
\begin{equation*}
\left.\begin{array}{r@{\ \ }c@{\ \ }ll}
{\partial^\alpha_t v_k}(\bar{x}^k,\bar{t}_k)&=&C_\alpha\jf_{-\infty}^{\bar{t}_k}\fr{v_k(\bar{x}^k,\bar{t}_k)-v_k(\bar{x}^k,\tau)}
{(\bar{t}_k-\tau)^{1+\alpha}}d\tau\leq 0\,. \\[0.2cm]
(-\Delta)^s v_k(\bar{x}^k,\bar{t}_k)&=&C_{n,s}P.V.\jf_{\R^{n}}\frac{v_k(\bar{x}^k,\bar{t}_k)-v_k(y,\bar{t}_k)}{\abs{\bar{x}^k-y}^{n+2s}}dy \\[0.3cm]
&=&C_{n,s}P.V.\jf_{\Sigma_\lambda}\frac{v_k(\bar{x}^k,\bar{t}_k)-v_k(y,\bar{t}_k)}{\abs{\bar{x}^k-y}^{n+2s}}dy+C_{n,s}\jf_{\Sigma_\lambda}\frac{v_k(\bar{x}^k,\bar{t}_k)-v_k(y^\lambda,\bar{t}_k)}{\abs{\bar{x}^k-y^\lambda}^{n+2s}}dy \\[0.3cm]
&\leq&C_{n,s}\jf_{\Sigma_\lambda}\frac{2v_k(\bar{x}^k,\bar{t}_k)-v_k(y,\bar{t}_k)-v_k(y^\lambda,\bar{t}_k)}{\abs{\bar{x}^k-y^{\lambda}}^{n+2s}}dy
\\[0.3cm]
&=&2C_{n,s}w_k(\bar{x}^k,\bar{t}_k)\jf_{\Sigma_\lambda}\frac{1}{\abs{\bar{x}^k-y^{\lambda}}^{n+2s}}dy
\\
&\leq&-\fr{Cm_k}{l^{2s}}.\end{array}\right.
\end{equation*}
It follows that
\begin{equation}\label{1.6}\partial^\alpha_t v_k(\bar{x}^k,\bar{t}_k)+(-\Delta)^s v_k(\bar{x}^k,\bar{t}_k)\leq-\fr{Cm_k}{l^{2s}}.\end{equation}
In addition, substituting $v_k$ into the differential equation in \eqref{1.1} and using the assumption $c(x,t)\leq C_0$, we obtain
\begin{equation}\label{1.7}
\partial^\alpha_t v_k(\bar{x}^k,\bar{t}_k)+(-\Delta)^s v_k(\bar{x}^k,\bar{t}_k)=c(\bar{x}^k,\bar{t}_k)w(\bar{x}^k,\bar{t}_k)-\varepsilon_k\partial^\alpha_t\eta_k(\bar{t}_k)\geq-C_0m-C\varepsilon_k.
\end{equation}
Then a combination of \eqref{1.6} and \eqref{1.7} yields that
\[-C_0m\leq-\fr{Cm_k}{l^{2s}}+C\varepsilon_k\to-\fr{Cm}{l^{2s}},\]
as $k\to\infty,$ which is a contradiction for sufficiently small $l$. Hence we complete the proof of \eqref{1.2}.

Next, we show the validity of \eqref{1.3}. If $w(x,t)$ vanishes  at $(x^0,t_0)\in\Omega\tm\R,$ then by \eqref{1.2}, we derive that
\[w(x^0,t_0)=\min\limits_{\Sigma_\lambda\tm\R}w(x,t)=0.\]
The equation in \eqref{1.1} obviously implies that
\begin{equation}\label{1.8}\partial^\alpha_t w(x^0,t_0)+(-\Delta)^s w(x^0,t_0)=0.
\end{equation}
On the other hand, since $w(x,t)\geq 0$ in $\Sigma_\lambda\tm\R$ and $$\abs{x^0-y^{\lambda}}>\abs{x^0-y}~\mb{provided~} y\in\Sigma_{\lambda},$$ we obtain
\begin{equation}\label{1.9}
\left.\begin{array}{r@{\ \ }c@{\ \ }ll}
(-\Delta)^s w(x^0,t_0)&=&C_{n,s}P.V.\jf_{\R^{n}}\frac{-w(y,t_0)}{\abs{x^0-y}^{n+2s}}dy \\[0.3cm]
&=&C_{n,s}P.V.\jf_{\Sigma_\lambda}w(y,t_0)\lt[\frac{1}{\abs{x^0-y^{\lambda}}^{n+2s}}-\frac{1}{\abs{x^0-y}^{n+2s}}\rt]dy \\[0.3cm]&\leq&0
\end{array}\right.
\end{equation}
and
\begin{equation}\label{1.10}
\left.\begin{array}{r@{\ \ }c@{\ \ }ll}
\partial^\alpha_t w(x^0,t_0)&=&C_\alpha\jf_{-\infty}^{t_0}\frac{-w(x^0,\tau)}{(t_0-\tau)^{1+\alpha}}d\tau\leq0.\\[0.3cm]
\end{array}\right.
\end{equation}
So it follows from \eqref{1.8}, \eqref{1.9} and \eqref{1.10} 
that $$0=\partial^\alpha_t w(x^0,t_0)=C_\alpha\jf_{-\infty}^{t_0}\frac{-w(x^0,\tau)}{(t_0-\tau)^{1+\alpha}}d\tau,$$ then we must have
\[w(x^0,\tau)\equiv 0=\min_{\Sigma_\lambda\tm\R}w(x,t), \mb{~for~}\forall \tau\in (-\infty,t_0],\]
that is, for each $\tau\in(-\infty,t_0]$, $w(x,t)$  attains zero at $(x^0,\tau)\in\Omega\tm\R$.

Now, repeating the previous process, we further obtain
\begin{equation*}\label{1.9}
\left.\begin{array}{r@{\ \ }c@{\ \ }ll}
0=(-\Delta)^s w(x^0,\tau)
&=&C_{n,s}P.V.\jf_{\Sigma_\lambda}w(y,\tau)\lt[\frac{1}{\abs{x^0-y^{\lambda}}^{n+2s}}-\frac{1}{\abs{x^0-y}^{n+2s}}\rt]dy.
\end{array}\right.
\end{equation*}
Together with the anti-symmetry of $w(y,\tau)$ with respect to $y$, we derive
\[w(y,\tau)\equiv 0\mb{~for~}\forall y\in \R^n.\]
Therefore, \[w(y,\tau)\equiv 0 \mb{~in~} \R^n\tm(-\infty,t_0].\]
This completes the proof of Theorem \ref{thm1}. \end{proof}

\subsection{Maximum principle in unbounded domains}
We now prove Theorem \ref{thm3},  the maximum principle for antisymmetric functions in unbounded domains. This is also  an ensential ingredient in proving the Liouville theorem for the dual fractional operator.
\begin{proof}[Proof of Theorem \ref{thm3}]
We argue by  contradiction. If \eqref{2.2} is not true,  since $w(x,t)$ is bounded from above in $\Sigma_\lm\times \R$, then there exists a constant $A>0$ such that
\begin{equation}\label{2.3}
\sup\limits_{(x,t)\in \Sigma_\lm\times \R}w(x,t):=A>0.
\end{equation}
Since the domain $\Sigma_\lm\times\R$ is unbounded, the supremum of $w(x,t)$ may not be attained in $\Sigma_\lm\times\R$, however, by \eqref{2.3}, there exists a maximizing sequence $\{(x^k,t_k)\}\subset \Sigma_\lm\times\R$ such that
\[w(x^k,t_k)\to A\ \mbox{as}\ k\to \infty.\]
More accurately, there exists a  sequence $\{\varepsilon_k\}\searrow 0$ such that
\begin{equation}\label{2.4}
w(x^k,t_k)=A-\varepsilon_k>0.
\end{equation}
Now we introduce a perturbation  of $w$ near $(x^k,t_k)$ as following
\begin{equation}\label{2.5}
v_k(x,t)=w(x,t)+\varepsilon_k\eta_k(x,t) \ \mbox{in}\ \R^n\times \R,
\end{equation}
where
\[\eta_k(x,t)=\eta\lt(\fr{x-x^k}{r_k/2},\fr{t-t_k}{(r_k/2)^{2s/\alpha}}\rt),\]
with $r_k=dist(x_k,T_{\lm})>0$ and 
  $\eta\in C^\infty_0(\R^n\times\R)$ is a cut-off smooth function satisfying
\begin{equation*}
\left\{\begin{array}{r@{\ \ }c@{\ \ }ll}
0\leq \eta\leq 1 &\mbox{in}&\ \ \R^n\times\R\,, \\[0.05cm]
 \eta= 1 &\mbox{in}&\ \ B_{1/2}(0)\times[-\fr{1}{2},\fr{1}{2}]\,, \\[0.05cm]
\eta= 0 &\mbox{in}&\ \ \left(\R^n\times\R\right) \backslash\left(B_{1}(0)\times[-1,1]\right)\,. \\[0.05cm]
\end{array}\right.
\end{equation*}
Denote
\[Q_k(x^k,t_k):=B_{r_k/2}(x^k)\times\lt[t_k-\lt(\fr{r_k}{2}\rt)^{2s/\alpha}, t_k+\lt(\fr{r_k}{2}\rt)^{2s/\alpha}\rt]\subset\Sigma_\lm\tm\R.\]
 By \eqref{2.3}, \eqref{2.4} and \eqref{2.5}, we have
\begin{equation*}
\left.\begin{array}{r@{\ \ }c@{\ \ }ll}
v_k(x^k,t_k)&=&A\,, \\[0.15cm]
v_k(x,t)=w(x,t)&\leq& A\ \mbox{in}\ \left(\Sigma_\lm\times\R\right) \backslash Q_k(x^k,t_k)\,, \\[0.15cm]
v_k(x,t)=\varepsilon_k\eta_k(x,t) &<& A\ \mbox{on}\ \ T_\lm\tm\R\,. \\[0.05cm]
\end{array}\right.
\end{equation*}
Since $w$ is upper semi-continuous on $\overline{\Sigma}_\lm\times \R$, then $v_k$ must attains its maximum value which is at least  $A$ at $\overline{Q_k(x^k,t_k)}\subset{\Sigma}_\lm\times \R$, that is, 
\begin{equation}\label{2.6}
\exists\ \{(\bar{{x}}^k,\bar{t}_k)\}\subset\overline{Q_k(x^k,t_k)}\ \   \ s.t.\ \  \   A+\varepsilon_k\geq v_k(\bar{{x}}^k,\bar{t}_k)=\sup\limits_{\Sigma_\lm\times\R}v_k(x,t)\geq A,
\end{equation}
where we have used  \eqref{2.3} and \eqref{2.5}.
Now, applying \eqref{2.6}, we derive
\begin{equation*}
\left.\begin{array}{r@{\ \ }c@{\ \ }ll}
w(\bar{x}^k,\bar{t}_k)&\geq&A-\varepsilon_k>0\,, \\[0.2cm]
{\partial^\alpha_t v_k}(\bar{x}^k,\bar{t}_k)&=&C_\alpha\jf_{-\infty}^{\bar{t}_k}\fr{v_k(\bar{x}^k,\bar{t}_k)-v_k(\bar{x}^k,\tau)}
{(\bar{t}_k-\tau)^{1+\alpha}}d\tau\geq 0\,. \\[0.2cm]
\end{array}\right.
\end{equation*}
Next, we  derive a contradiction by estimating the value of $(-\Delta)^s v_k$ at the maximum point $(\bar{x}^k,\bar{t}_k)$ of $v_k$ in $\Sm_\lm\tm\R.$
 On one hand,
taking into account of differential inequality in \eqref{2.1}, \eqref{2.5} and translation and scaling invariance of the operator $\partial_t^\alpha+(-\Delta)^s$(see \eqref{TS}, we obtain
\begin{eqnarray}\nonumber\label{2.7}
(-\Delta)^s v_k(\bar{x}^k,\bar{t}_k)&=&(-\Delta)^s w(\bar{x}^k,\bar{t}_k)+\varepsilon_k(-\Delta)^s \eta_k(\bar{x}^k,\bar{t}_k)
 \\  \nonumber
&\leq& -\partial^\alpha_t w(\bar{x}^k,\bar{t}_k)+\varepsilon_k(-\Delta)^s \eta_k(\bar{x}^k,\bar{t}_k)
 \\ \nonumber
&\leq&\varepsilon_k\lt[ {\partial^\alpha_t \eta_k} (\bar{x}^k,\bar{t}_k)+(-\Delta)^s  \eta_k(\bar{x}^k,\bar{t}_k)\rt]
\\
&\leq& C\fr{\varepsilon_k }{r_k^{2s}}.
\end{eqnarray}
On the other hand, starting from the definition of operator $(-\Delta)^s$ and utilizing the antisymmetry of
$w$ in $x$ as well as  the fact $\abs{\bar{x}^k-y^\lm}>\abs{\bar{x}^k-y}$ and \eqref{2.6}, we compute
\begin{equation}\label{2.8}
\left.\begin{array}{r@{\ \ }c@{\ \ }ll}
(-\Delta)^s v_k(\bar{x}^k,\bar{t}_k)&=&C_{n,s}P.V.\jf_{\R^{n}}\frac{v_k(\bar{x}^k,\bar{t}_k)-v_k(y,\bar{t}_k)}{\abs{\bar{x}^k-y}^{n+2s}}dy \\[0.3cm]
&=&C_{n,s}P.V.\jf_{\Sigma_\lambda}\frac{v_k(\bar{x}^k,\bar{t}_k)-v_k(y,\bar{t}_k)}{\abs{\bar{x}^k-y}^{n+2s}}dy+C_{n,s}\jf_{\Sigma_\lambda}\frac{v_k(\bar{x}^k,\bar{t}_k)-v_k(y^\lambda,\bar{t}_k)}{\abs{\bar{x}^k-y^\lambda}^{n+2s}}dy \\[0.3cm]
&\geq&C_{n,s}\jf_{\Sigma_\lambda}\frac{2v_k(\bar{x}^k,\bar{t}_k)-v_k(y,\bar{t}_k)-v_k(y^\lambda,\bar{t}_k)}{\abs{\bar{x}^k-y^{\lambda}}^{n+2s}}dy
\\[0.3cm]
&\geq&C_{n,s}2\lt(v_k(\bar{x}^k,\bar{t}_k)-\ve_k\rt)\jf_{\Sigma_\lambda}\frac{1}{\abs{\bar{x}^k-y^{\lambda}}^{n+2s}}dy
\\
&\geq&\fr{C(A-\ve_k)}{r_k^{2s}}.\end{array}\right.
\end{equation}
Finally, a combination of \eqref{2.7} and \eqref{2.8} yields that
\[A-\ve_k\leq C\ve_k,\]
which leads to a contradiction for sufficiently large $k$. Hence we conclude that \eqref{2.2} is valid.
\end{proof}

\section{Radial symmetry of solutions}
In this section, we employ the narrow region principle (Theorem \ref{thm1}) as a fundamental tool to initiate the direct moving plane method, then by combining perturbation techniques and  limit arguments,  for the dual fractional equation
\begin{equation*}
\partial^\alpha_t u(x,t)+(-\Delta)^s u(x,t) =  f(u(x,t))\ \ \mbox{in}\ \ B_1(0)\times\R ,
\end{equation*}
 under suitable assumptions on the nonlinear term $f$, we show that the solution $u(\cdot,t)$ with the vanishing exterior condition is radially symmetric and strictly  decreasing with respect to the origin in a unit ball.
 \begin{proof}[Proof of Theorem \ref{thm2}]
Let $x_1$ be any direction and for any $\lm\in\R,$ we define $T_\lambda,\ \Sigma_{\lambda},\ \Om_{\lm},\ x^{\lm},\ w_\lm$ as described in  section 1. Substituting the definition of $w_\lm$ into the equation \eqref{1.11}, we have
\begin{equation}\label{1.12}
\left\{
\begin{array}{ll}
    \partial^\alpha_t w_\lm(x,t)+(-\Delta)^s w_\lm(x,t)=c_\lm(x,t)w_{\lm}(x,t) ,~   &(x,t) \in  \Omega_\lm\times\mathbb{R}  , \\
  w_\lm(x,t)\geq 0 , ~ &(x,t)  \in (\Sigma_\lambda\backslash\Omega_\lm)\times\mathbb{R},\\
  w_\lm(x,t)=-w_\lm(x^\lambda,t) , ~ &(x,t)  \in \Sigma_\lambda\times\mathbb{R}.
\end{array}
\right.
\end{equation}
where the weighted function $$c_\lm(x,t)=\fr{f(u_\lm(x,t))-f(u(x,t))}{u_\lm(x,t)-u(x,t)}$$ is bounded in $\Om_\lm\tm\R$ due to $f\in C^1\lt([0,+\infty)\rt).$ Now we carry out the direct method of moving plane which is  devided into two steps as outlined below.

$\rm{\mathbf{Step~1.}}$ Start moving the plane $T_\lm$ from $x_1=-1$ to the right along the $x_1$-axis.

 When $\lm$ is sufficiently closed to $-1$, $\Om_\lm$ is a narrow region. Then by applying the narrow rigion principle, Theorem \ref{thm1}, to
problem \eqref{1.12}, we deduce that
\begin{equation} \label{1.13}
w_\lm(x,t)\geq 0 \mb{~in~} \Sm_\lm\tm\R.
\end{equation}
This provides a starting point to move  the plane $T_\lm.$

$\rm{\mathbf{Step~2.}}$ Continuing  to  move the plane $T_\lm$  towards the right along the $x_1$-axis until  reaching its limiting position as long as  inequality \eqref{1.13}  holds. Denote
\[\lm_0:=\sup\{\lm<0\mid w_\mu(x,t)\geq 0, (x,t)\in\Sm_{\mu}\tm\R \mb{~for~any~}\mu\leq \lm\}.\]
We are going to employ the contradiction argument to  verify that
\begin{equation} \label{1.14}
\lm_0=0.
\end{equation}
Otherwise, if $\lm_0<0,$ according to the definition of $\lm_0$, there exsits a sequences of negative numbers $\{\lm_k\}$ with $\{\lm_k\}\searrow \lm_0$ and a sequence of positive numbers $\{m_k\}\searrow0$  such that
\[\inf\limits_{\Om_{\lm_k}\tm\R}w_{\lm_k}(x,t)=\inf\limits_{\Sm_{\lm_k}\tm\R}w_{\lm_k}(x,t)=-m_k .\]
 It implies that  for each fixed $k>0,$ there exists a point $(x^{k},t_{k}) \in\Om_{\lm_k}\tm\R$ 
 such that
\[-m_k\leq w_{\lm_k}(x^k,t_k)= -m_k+m_k^2<0.\]

Since $\R$ is an unbounded interval, the infimum of $w_{\lm_k}$
with respect to $t$ may not be attained. In order to estimate $\partial_t^\alpha w_{\lm_k}$,  we need to introduce a perturbation of $w_{\lm_k}$ near $t_k$ as follows
\begin{equation}\label{1.15}
v_{k}(x,t)=w_{\lm_k}(x,t)-m_k^2\eta_k(t) \ \mbox{in}\ \Sigma_{\lm_k}\times \R,
\end{equation}
where $\eta_k(t)=\eta(t-t_k)$ with $\eta\in C_0^\infty(-1,1)$ be a cut-off function as in the proof of Theorem \ref{thm1}.
Based on the above analysis and the exterior condition in \eqref{1.12} satisfied by $w_{\lm_k}$, we have
\begin{equation*}
\left\{\begin{array}{r@{\ \ }c@{\ \ }ll}
v_{k}(x^k,t_k)&=&-m_k\,, \\[0.15cm]
v_{k}(x,t)=w_{\lm_k}(x,t)&\geq& -m_k\ \mbox{in}\ \Omega_{\lm_k}\tm(\R \backslash (-1+t_k,1+t_k))\,, \\[0.15cm]
v_{k}(x,t)\geq-m_k^2\eta_k(x,t) &>& -m_k\ \mbox{in}\ (\Sigma_{\lambda_k}\backslash\Omega_{\lm_k})\times \R\,. \\[0.05cm]
\end{array}\right.
\end{equation*}
Since $u$ is continuous on $\overline\Omega_{\lm_k}\times \R$, then $v_{k}$ must attains its minimum value which is at most  $-m_k$ at $\Omega_{\lm_k}\times (-1+t_k,1+t_k)$, that is,
\begin{equation*}\label{1.16}
\exists\ \{(\bar{{x}}^k,\bar{t}_k)\}\subset\Omega_{\lm_k}\times (-1+t_k,1+t_k)\   \ s.t.\ \     -m_k-m_k^2\leq v_{k}(\bar{{x}}^k,\bar{t}_k)=\inf\limits_{\Sigma_{\lambda_k}\times\R}v_{k}(x,t)\leq -m_k,
\end{equation*}
which implies that
\begin{equation}\label{1.16}-m_k\leq w_{\lm_k}(\bar{x}^k,\bar{t}_k)\leq-m_k+m_k^2<0.\end{equation}
Similar to the process of Theorem \ref{thm1}, we have
\begin{equation}\label{1.24}\left.
\begin{array}{ll}\partial^\alpha_t v_{k}(\bar{x}^k,\bar{t}_k)+(-\Delta)^s v_{k}(\bar{x}^k,\bar{t}_k)&\leq 2C_{n,s}w_{\lm_k}(\bar{x}^k,\bar{t}_k)\jf_{\Sigma_{\lambda_k}}\frac{1}{\abs{\bar{x}^k-y^{\lambda_k}}^{n+2s}}dy
\\[0.3cm]&\leq-\fr{C(m_k-m_k^2)}{{dist}(\bar{x}^k,T_{\lm_k})^{2s}}
.\end{array}
\rt.\end{equation}
Furthermore, it follows from the  differential equation in \eqref{1.12} and \eqref{1.16} 
that
\begin{equation*}
\left.
\begin{array}{ll}
\partial^\alpha_t v_{k}(\bar{x}^k,\bar{t}_k)+(-\Delta)^s v_{k}(\bar{x}^k,\bar{t}_k)&=c_{\lm_k}(\bar{x}^k,\bar{t}_k)w_{\lm_k}(\bar{x}^k,\bar{t}_k)-m_k^2\partial^\alpha_t\eta_k(\bar{t}_k)
\\
&\geq-c_{\lm_k}(\bar{x}^k,\bar{t}_k)m_k-Cm_k^2.
\end{array}
\rt.
\end{equation*}
Here we may assume $c_{\lm_k}(\bar{x}^k,\bar{t}_k)\geq 0$ without loss of generality. Otherwise, a contradiction can be derived from \eqref{1.24}. Consquently,
\begin{equation}\label{dist}-c_{\lm_k}(\bar{x}^k,\bar{t}_k)-Cm_k\leq -\fr{C(1-m_k)}{{dist}(\bar{x}^k,T_{\lm_k})^{2s}}
\leq-\fr{C(1-m_k)}{2^{2s}},\end{equation}
by virtue of $m_{k}\to 0$ as $k\to\infty$,  we derive that for  sufficiently large $k$,
\[c_{\lm_k}(\bar{x}^k,\bar{t}_k)\geq C_0>0.\]
This implies that there exists  $\mb{~some~}\xi_{k}\in\lt(u_{\lm_k}(\bar{x}^k,\bar{t}_k),u(\bar{x}^k,\bar{t}_k)\rt) $ such that \[f'(\xi_{k})\geq C_0. \]
Thus, owing to \eqref{1.16} and the assumption $f'(0)\leq 0$, after extracting a subsequence, we  obtain
\begin{equation}\label{1.17}u(\bar{x}^k,\bar{t}_k)\geq C_1>0,\end{equation}
for sufficiently large  $k$.

In order to simplify the notation, we denote
\[\tilde{w}_{k}(x,t)=w_{\lm_k}(x,t+\bar{t}_k)\mb{~and~}\tilde{c}_{k}(x,t)=c_{\lm_k}(x,t+\bar{t}_k).\]
It follows from Arzel$\grave{a}$-Ascoli theorem that there exist two continuous function $\tilde{w}$ and $\tilde{c}$ such that

\[
\lim_{{k \to \infty}} \tilde{w}_{k}(x, t) =  \tilde{w}(x, t)
\]
and
\[
\lim_{{k \to \infty}} \tilde{c}_{k}(x, t) = \tilde{c}(x, t)
\]
uniformly in $B_1(0)\tm\R.$

Moreover, taking into account of the equation
\[
{\partial_t^\alpha}\tilde{w}_{k}(x, t) +(- \Delta)^s \tilde{w}_{k}(x, t) = \tilde{c}_{k}(x, t)\tilde{w}_{k}(x, t), \quad \text{in } \Omega_{\lm_k} \times \mathbb{R},
\]
we  conclude that the limit function $\tilde{w}$ satisfies
\begin{equation}\label{1.18}
{\partial_t^\alpha}\tilde{w}(x, t) +(- \Delta)^s \tilde{w}(x, t) = \tilde{c}(x, t)\tilde{w}(x, t), \quad \text{in } \Omega_{\lm_0} \times \mathbb{R}.
\end{equation}
As mentioned in \eqref{dist}, combining  the uniform boundedness of $c_{\lm_k}(\bar{x}^k,\bar{t}_k)$ with $\Omega_{\lm_k}\subset B_1(0)$ and  $\lm_k\to\lm_0,$ we may assume that $\bar{x}^k\to x^0\in {\Sigma}_{\lm_0}\cap \overline{B}_1(0).$
Then applying \eqref{1.16} and the continuity on $u$, we obtain
 \begin{equation}\label{mini}\tilde{w}(x^0, 0)=0=\inf\limits_{\Sigma_{\lm_0}\tm\R}{w}_{\lm_0}(x,t)=\inf\limits_{\Sigma_{\lm_0}\tm\R}\tilde{w}(x,t).\end{equation}
 Substituting this into the limit equation \eqref{1.18}, it yields
 \begin{equation*}\left.
\begin{array}{ll}0&={\partial_t^\alpha}\tilde{w}(x^0, 0) +(- \Delta)^s \tilde{w}(x^0,0)\\[0.3cm]
&=C_\alpha\jf_{-\infty}^0\fr{-\tilde{w}(x^0,\tau)}{(-\tau)^{1+\alpha}}d\tau +C_{n,s}P.V.\jf_{\Sigma_{\lambda_0}}\tilde{w}(y,0)\lt[\frac{1}{\abs{x^0-y^{\lambda}}^{n+2s}}-\frac{1}{\abs{x^0-y}^{n+2s}}\rt]dy.\end{array}
\rt.\end{equation*}
As a result of \eqref{mini}, the antisymmetry of $\tilde{w}(x, t)$ with respect to $x$ and the fact that $\abs{x^0-y^{\lambda}}>\abs{x^0-y}$, we  conclude
\begin{equation}\label{1.19}\tilde{w}(x, t) \equiv 0,~~(x, t) \in \mathbb{R}^n \times (-\infty, 0].
\end{equation}

Correspondingly, we define
\[u_k(x, t) = u(x, t + \bar{t}_k).\]
Similar to the previous discussion regarding $\tilde{w}_{k}$, we also have
\[\lim_{k \to \infty} u_k(x, t) = \tilde{u}(x, t),\]
and
\begin{equation}\label{1.20}
{\partial_t^\alpha}\tilde{u}(x, t) +(- \Delta)^s \tilde{u}(x, t) = f\lt(\tilde{u}(x, t)\rt) \quad \text{in } B_1(0) \times \mathbb{R}.
\end{equation}
In addition, by using \eqref{1.17}, we  infer that
\begin{equation}\label{1.21}\tilde{u}(x^0, 0) = \lim_{j \to \infty} u(\bar{x}^j, \bar{t}_j)\geq C_1 > 0.\end{equation}

Next, we will show that \begin{equation}\label{1.22}\tilde{u}(x,0)> 0~\mb{in}~B_1(0).\end{equation}
If this is not true, according to the exterior condition and the interior positivity of $u$, then there exists a point $\bar{x} \in B_1(0)$ such that $$\tilde{u}(\bar{x}, 0) =\inf\limits_{\R^n\tm\R}\tilde{u}(x,t)= 0,$$  which, together with  limit equation \eqref{1.20} and the assumption $f(0)\geq0$, leads to
 \begin{equation*}0=(- \Delta)^s \tilde{u}(\bar{x},0)=C_{n,s}P.V.\jf_{\R^n}\frac{-\tilde{u}(y,0)}{\abs{\bar{x}-y}^{n+2s}}dy.\end{equation*}
Thus, $\tilde{u}(x,0)\equiv 0\mb{~in~} \R^n$ due to $u\geq 0.$  This contradicts \eqref{1.21} and thus verifies the assersion \eqref{1.22}.

Due to the condition $\tilde{u}(x, 0) \equiv 0$ in $B_{1}^{c}(0)$, \eqref{1.22}  and $\lambda_0 < 0$, we further conclude that there must exists a point $\tilde{x} \in B_{1}^{c}(0)$ such that $\tilde{x}^{\lambda_0} \in B_1(0)$ and $$\tilde{w}(\tilde{x}, 0) = \tilde{u}(\tilde{x}^{\lambda_0}, 0) - \tilde{u}(\tilde{x}, 0) = \tilde{u}(\tilde{x}^{\lambda_0}, 0) > 0.$$ However, this contradicts \eqref{1.19}. Hence, we have established that the limiting position must be $T_0$.

By choosing  $x_1$ arbitrarily and considering the definition of $\lambda_0$, we  deduce that $u(\cdot,t)$ must be  radially symmetric and  monotone nonincreasing about the origin  in the unit ball $B_1(0)$. Now we are ready to demonstrate the strict monotonicity,  more specifically, it is sufficient to prove that
\begin{equation}\label{1.23}w_{\lm}(x,t)>0, ~\forall\lm\in(-1,0).
\end{equation}
If  not, then there exists some $\lambda_0 \in (-1, 0)$ and a point $(x^0, t_0) \in \Omega_{\lambda_0} \times \mathbb{R}$ such that $$w_{\lambda_0}(x_0, t_0)=\min\limits_{\Sigma_{\lm_0}\tm\R} w_{\lambda_0}= 0.$$
Combining the differential equation in \eqref{1.12} with the definition of the dual fractional operator $\partial_t^{\alpha} + (-\Delta)^s$, similar to the previous argument, we must have
\[w_{\lambda_0}(x, t)
\equiv 0 \mb{~in~} \Sigma_{\lm_0}\tm(-\infty,t_0].\]
 This is a contradiction due to the fact that $u(\cdot,t)>0$ in $B_1(0)$ and $u(\cdot,t)\equiv 0$ in $B_1^c(0)$ for each fixed $t\in\R.$
 Hence, we verify  the assertion \eqref{1.23}  and thus complete the proof of
Theorem \ref{thm2}.
\end{proof}
\section{Liouville Theorem}
In this section, we begin by employing perturbation techniques and analyzing the nonlocal one-sided nature of the one-dimensional operator $\partial^\alpha_t$ to establish the Liouville theorem for the Marchaud fractional time operator $\partial^\alpha_t$, Theorem \ref{thm5}. Directly following this, by incorporating the maximum principle in unbounded domain as stated in Theorem \ref{thm3}, we  will be able to derive    our second main result, Theorem \ref{thm4}.
\subsection{Liouville Theorem for the Marchaud fractional time operator $\partial^\alpha_t$}
Let us begin by recalling the definition of the Marchaud derivitive \begin{equation}\label{M}\partial_t^\alpha u(t)=C_\alpha\jf_{-\infty}^{t}\fr{u(t)-u(\tau)}{(t-\tau)^{1+\alpha}}d\tau.\end{equation}
 Now we  show that a  bounded solution of equation $\partial_t^\alpha u(t)=0$ in $\R^n$ must be constant.
\begin{proof}[Proof of Theorem \ref{thm5}]
The proof goes by contradiction. Since $u(t)$ is bounded in $\R,$ we may assume that
\begin{equation}\label{max}M:=\sup\limits_{t\in\R}u(t)>\inf\limits_{t\in\R}u(t)=:m.\end{equation}
 Now we divide the proof into three cases based on whether the maximum and minimum values are attained and proceed to derive a contradiction for each case.

 $\mathbf{Case ~1} $: The extrema (maximum and minimum) of \(u\) are both attained in $\R$.

  Suppose that u attains its maximum at $\bar{t}$ and its minimum at $\underline{t}$ with $\underline{t}<\bar{t}.$ Owing to equation \eqref{anti1} and the  nonlocal one-sided nature of $\partial_t^\alpha,$ see \eqref{M},
   we have
  $$u(t)\equiv u(\underline{t})=m ~\mb{for}~t<\underline{t}$$and $$u(t)\equiv u(\bar{t})=M ~\mb{for}~t<\bar{t}.$$
  This contracdicts the assumption $\underline{t}<\bar{t}$. We can derive a similar contradiction in the case $\overline{t}<\underline{t}$.

 $\mathbf{Case ~2}$: Only one of the extrema (maximum or minimum) of \(u\) is attained in $\R$.

 Without loss of generality, we may assume that \(u\) attains its maximum at $t_0$ and there exists a minimizing sequence $\{t_k\}\searrow-\infty$ such that \begin{equation}\label{case2}\lim\limits_{k\to\infty}u(t_k)=m.\end{equation}
 Then applying  equation \eqref{anti1} and the  definition of $\partial_t^\alpha$  \eqref{M},
   we have
  $$u(t)\equiv u(t_0)=M ~\mb{for}~t<t_0,$$ which contradicts \eqref{case2} due to the continuity of $u.$

 $\mathbf{Case ~3}$: The extrema (maximum and minimum) of \(u\) are both unattainable.

We assume without loss of generality that there exist a minimizing sequence $\{\underline{t}_k\}\searrow-\infty$ and a maximizing sequence $\{\bar{t}_k\}\searrow-\infty$ 
and a sequence $\{\ve_k\}\searrow0$  such that
$$u(\overline{t}_k)=M-\ve_k$$
and $$u(\underline{t}_k)=m+\ve_k.$$
By extracting subsequences, we may assume $\overline{t}_k-\underline{t}_k>1.$

Now we introduce a perturbation  of $w$ near $\underline{t}_k$ and $\overline{t}_k$ as following
\begin{equation*}
v_k(t)=u(t)+\varepsilon_k\eta_k(t) \ \mbox{in}\  \R,
\end{equation*}
where
\[\eta_k(t)=\eta\lt(\frac{t-\overline{t}_k}{r_k}\rt)-\eta\lt(\frac{t-\underline{t}_k}{r_k}\rt),\]
with $r_k=\frac{1}{4}(\overline{t}_k-\underline{t}_k)>0$ and
  $\eta\in C^\infty_0(\R)$ is a cut-off smooth function as described in the proof of Theorem\ref{thm1}.
  Clearly, $supp \eta_k\subset (-r_k+\underline{t}_k,r_k+\underline{t}_k)\cup(-r_k+\overline{t}_k,r_k+\overline{t}_k)$ and there holds $$\eta_k(\overline{t}_k)=1,~ \eta_k(\underline{t}_k)=-1,$$
  \[\eta_k(t)=-\eta\lt(\frac{t-\overline{t}_k}{r_k}\rt)\leq 0~\mb{in}~\R\backslash(-r_k+\overline{t}_k,r_k+\overline{t}_k)\]and
  \[\eta_k(t)=\eta\lt(\frac{t-\underline{t}_k}{r_k}\rt)\geq 0~\mb{in}~\R\backslash(-r_k+\underline{t}_k,r_k+\underline{t}_k).\]
  Then   we have
\begin{equation*}
\left\{\begin{array}{r@{\ \ }c@{\ \ }ll}
v_k(\overline{t}_k)&=&M,~v_k(\underline{t}_k)=m\,, \\[0.1cm]
v_k(t)&\leq& M \ ~\mb{in}~\R\backslash(-r_k+\overline{t}_k,r_k+\overline{t}_k)\,, \\[0.1cm]
v_k(t)&\geq& m\ \ ~\mb{in}~\R\backslash(-r_k+\underline{t}_k,r_k+\underline{t}_k)\,. \\[0.05cm]\end{array}\right.
\end{equation*}
Subsequently, $v_k$ must attain its maximum value, which is at least  $M$, at $[-r_k+\overline{t}_k,r_k+\overline{t}_k]$ and also attain its minimum value, which is at most  $m$, at $[-r_k+\underline{t}_k,r_k+\underline{t}_k]$, more specifically, 
\begin{equation*}
\exists\ \{\bar{s}_k\}\subset[-r_k+\overline{t}_k,r_k+\overline{t}_k]\ \   \ s.t.\ \  \   M+\varepsilon_k\geq v_k(\bar{s}_k)=\sup\limits_{t\in\R}v_k(t)\geq M.
\end{equation*}and
\begin{equation*}
\exists\ \{\underline{s}_k\}\subset[-r_k+\underline{t}_k,r_k+\underline{t}_k]\ \   \ s.t.\ \  \   m-\varepsilon_k\leq v_k(\underline{s}_k)=\inf\limits_{t\in\R}v_k(t)\leq m.
\end{equation*}
Consequently,
\begin{equation}\label{2.12}\left.\begin{array}{r@{\ \ }c@{\ \ }ll}\partial^\alpha_t v_k(\bar{s}_k)&=&C_\alpha\jf_{-\infty}^{\bar{s}_k}\fr{v_k(\bar{s}_k)-v_k(\tau)}
{(\bar{s}_k-\tau)^{1+\alpha}}d\tau\\[0.3cm]
&\geq&C_\alpha\jf_{-\infty}^{\underline{s}_k}\fr{v_k(\bar{s}_k)-v_k(\tau)}
{(\bar{s}_k-\tau)^{1+\alpha}}d\tau
\\[0.4cm]
&=&C_\alpha\lt\{\jf_{-\infty}^{\underline{s}_k}\fr{v_k(\bar{s}_k)-v_k(\underline{s}_k)}
{(\bar{s}_k-\tau)^{1+\alpha}}d\tau+
\jf_{-\infty}^{\underline{s}_k}\fr{v_k(\underline{s}_k)-v_k(\tau)}
{(\bar{s}_k-\tau)^{1+\alpha}}d\tau\rt\}
\\[0.4cm]
&\geq&C_\alpha\lt\{(M-m)\jf_{\underline{s}_k-r_k}^{\underline{s}_k}\fr{1}
{(\overline{s}_k-\tau)^{1+\alpha}}d\tau+\jf_{-\infty}^{\underline{s}_k}\fr{v_k(\underline{s}_k)-v_k(\tau)}
{(\underline{s}_k-\tau)^{1+\alpha}}d\tau\rt\}
\\[0.4cm]
&\geq&\fr{C_0}
{{r_k}^\alpha}+\partial^\alpha_t v_k(\underline{s}_k).
\end{array}\right.
\end{equation}
In addition, owing to the equation in \eqref{anti1}, we ultilize the  rescaling and translation  for $\partial_t^\alpha\eta$  (see \eqref{TS}), it is easily derived
\begin{equation}\label{2.13}\partial^\alpha_t v_k(\bar{s}_k),\ \partial^\alpha_t v_k(\bar{s}_k^\lm)\sim{\fr{\ve_k}{{r_k}^\alpha}}.\end{equation}
It follows from \eqref{2.12} and \eqref{2.13} that
\[C\ve_k\geq C_0-C \ve_k,\]
which leads to a contradiction for sufficiently large $k$.

In conclusion, we  verifies  \eqref{max} and thus completes the  proof of Theorem \ref{thm5}.
\end{proof}
\subsection{Liouville Theorem  for the dual fractional operator $\partial^\alpha_t +(-\Delta)^s$}
In the rest of this section,  we employ the Maximum principle  (Theorem \ref{thm5}) for antisymmetric functions in unbounded domains,  along with the Liouville theorem for the Marchaud fractional time operator $\partial^\alpha_t$ just established  in Section 4.1, to complete the proof of the Liouville theorem (Theorem \ref{thm4}) for the dual fractional operator $\partial^\alpha_t +(-\Delta)^s.$
\begin{proof}[Proof of Theorem \ref{thm4}]
For each fixed $t\in\R$, we first claim that $u(\cdot, t)$ is symmetric with respect to any hyperplane in $\R^n$. Let $x_1$ be any given direction in $\R^n$, and we keep the notation $T_\lm,\Sm_\lm,w_\lm(x,t),$ $u_\lm(x,t),$ $x^\lm$ defined in  section 1. For any $\lm\in\R$, on account of equation \eqref{2.9}, we derive
\begin{equation*}
\left\{\begin{array}{ll}
\partial^\alpha_t w_\lm(x,t)+(-\Delta)^s w_\lm(x,t) =0,\ \ &\mbox{in}\ \ \Sm_\lm\times\R ,\\[0.3cm]
w_\lm(x,t)=-w_\lm(x^\lm,t), &\mbox{in}\ \ \Sm_\lm\times\R.
\end{array}\right.\end{equation*}
It follows from Theorem \ref{thm3} that
\[w_\lm(x,t)\equiv 0 \mb{~in~} \Sm_\lm\tm\R.\]
As a result,  the arbitrariness of $\lm$  indicates that $u(\cdot, t)$ exhibits symmetry with respect to any hyperplane perpendicular to the $x_1$-axis. Moreover, since the selection of the $x_1$ direction is arbitrary, we  conclude that $u(\cdot, t)$ is symmetric with respect to any hyperplane in $\R^n$ for each fixed $t\in\R$. Thus, we  deduce that $u(x, t)$ depends only on $t$, i.e.,
\[u(x,t)=u(t) \mb{~in~}\R^n\tm\R.\]

Now equation \eqref{2.9} reduce to the following one-dimensional one-sided fractional equation
\begin{equation*}
\partial^\alpha_t u(t) = 0\ \ \mbox{in}\ \ \R .
\end{equation*}
Then  Theorem \ref{thm5} yields that  $u(t)$ must be constant.  Thus, we have confirmed that the bounded solution of  equation \eqref{2.9} must be  constant. This completes the  proof of Theorem \ref{thm4}.
\end{proof}
\vskip4mm
$\mathbf{Acknowledgement}$
The work of 
the second author is partially supported by the National Natural Science Foundation of China (NSFC Grant No.12101452) and the work of the third author is partially supported by the National Natural Science Foundation of China (NSFC Grant No.12071229).


\begin{thebibliography}{99}
\bibitem[1]{1} M. Allen, L. Caffarelli,  A. Vasseur, A parabolic problem with a fractional time derivative. Arch. Ration. Mech. Anal., 221(2)(2016), 603-630.
\bibitem[2]{2} M. Allen, L. Caffarelli,  A. Vasseur,  Porous medium flow with both a fractional potential pressure and fractional time derivative. Chin. Ann. Math., 38B(1),(2017), 45-82.






\bibitem[3]{HB3}H. Berestycki, L. Nirenberg, On the method of moving planes and the sliding method, Bol. Soc. Bras. Mat, 22 (1991), 1-37.
\bibitem[4]{BK} K. Bogdan. T. Kulczycki, and A. Nowak, Gradient estimates for harmonic and q-harmonic functions of symmetric
stable processes, Illinois J. Math., 46 (2002), 541-556.

\bibitem[5]{14}L. Caffarelli, L. Silvestre, An extension problem related to the fractional Laplacian, Comm. Partial Differential Equations, 32 (2007), 1245-1260.


\bibitem[6]{Chang}S.-Y. A. Chang,  M. d. M. Gonz$\grave{a}$lez, Fractional Laplacian in conformal geometry, Adv. Math, 226 (2011), no. 2, 1410-1432.
  \bibitem[7]{Chang1}S.-Y. A. Chang,  Y. Yuan, A Liouville problem for the Sigma-2 equation, Discrete and Continuous Dynamical Systems, 28(2010), no. 2, 659-664.
\bibitem[8]{10}W. Chen, Y. Hu, and L. Ma, Moving planes and sliding methods for fractional elliptic and parabolic equations,  Adv. Nonlinear Stud., 2023, inpress.
\bibitem[9]{CL} W. Chen, C. Li, Maximum principles for the fractional p-Laplacian and symmetry of solutions, Adv. Math., 335
(2018), 735-758.

\bibitem[10]{CLL} W. Chen, C. Li, G. Li, Maximum principles for a fully nonlinear fractional order equation and symmetry of
solutions, Calc. Var., 56 (2017), 29.
\bibitem[11]{11} W. Chen, C. Li, and Y. Li, A direct method of moving planes for the fractional Laplacian, Adv. Math., 308 (2017), 404-437.


\bibitem[12]{13}W. Chen, C. Li, and J. Zhu, Fractional equations with indefinite nonlinearities, Disc. Cont. Dyn. Sys., 39 (2019), 1257-1268.
\bibitem[13]{cm} W. Chen, L. Ma, Qualitative properties of solutions for dual fractional nonlinear parabolic equations,  J. Funct. Anal., 2023, accepted.




\bibitem[14]{CM}W. Chen, L. Ma, Gibbons' conjecture for entire solutions of master equations, arXiv:2304.07888v1, 2023.


\bibitem[15]{18}W. Chen, L. Wu, Liouville theorems for fractional parabolic equations, Adv. Nonlinear Stud., 21 (2021), 939-
958.
\bibitem[16]{CW}W. Chen, L. Wu, Uniform a priori estimates for solutions of higher critical order fractional equations, Calc. Var.,
60 (2021), 102.

\bibitem[17]{Sun1}N. Cui, H. Sun, Ground state solution for a nonlinear fractional magnetic Schr\"odinger equation with indefinite potential, J. Math. Phys. 63 (2022), 091510.
\bibitem[18]{20}W. Dai, Z. Liu, and G. Lu, Liouville type theorems for PDE and IE systems involving fractional Laplacian on a half space, Potential Anal., 46 (2017), 569-588.
\bibitem[19]{DQ} W. Dai, G. Qin, Liouville-Type theorems for fractional and higher-order
h$\acute{e}$non-hardy type equations via the method of scaling spheres, Int. Math. Res. Not., 11 (2023),  9001-9070.
\bibitem[20]{23} W. Dai, G. Qin, and D. Wu, Direct Methods for pseudo-relativistic Schr\"odinger operators, J. Geom. Anal., 31
(2021), 5555-5618.
\bibitem[21]{dCN}  D. del-Castillo-Negrete, B. A. Carreras, V. E. Lynch, Nondiffusive transport in plasma turbulene: a fractional
diffusion approach, Phys. Rev. Lett., 94 (2005), 065003.
\bibitem[22]{F} A. Farina, Liouville-type theorems for elliptic problems, Handbook of Differential Equations: Stationary Partial
Differential Equations, 4 (2007), 61-116.

 \bibitem[23]{24}Y. Guo, L. Ma, Z. Zhang, Sliding methods for dual fractional nonlinear divergence type
parabolic equations and the Gibbons' conjecture, submitted to Adv. Nonlinear Stud., 2023.
\bibitem[24]{28} H. Hopf, Lectures on differential geometry in the large, Stanford University, 1956
\bibitem[25]{JW} S. Jarohs, T. Weth, Asymptotic symmetry for a class of nonlinear fractional reaction-diffusion equations, Discrete
Contin. Dyn. Syst., 34 (2014), 2581-2615.
\bibitem[26]{Sun2}Z. Jin, H. Sun, J. Zhang,  Existence of ground state solutions for critical fractional Choquard equations involving periodic magnetic field, Adv. Nonlinear Stud. 22 (2022), 372-389.
 \bibitem[27]{LC2}   C. Li, L. Wu, Pointwise regularity for fractional equations, J. Differential Equations, 302 (2021), 1-36.
   \bibitem[28]{LC3}  C. Li, Z. Wu, Radial symmetry for systems of fractional Laplacian, Acta Math. Sci. Ser. B (Engl. Ed.) 38 (2018), no. 5, 1567-1582.
\bibitem[29]{31}  J. Li, G. Lu, and J. Wang, Potential characterizations of geodesic balls on hyperbolic spaces: a moving plane
approach, J. Geom. Anal., 33 (2023), 134.
\bibitem[30]{33} R. Metzler, J. Klafter, The random walk's guide to anomalous diffusion: A fractional dynamics approach, Phys. Rep., 339 (2000), 1-77.
\bibitem[31]{MGZ} L. Ma, Y. Guo, and Z. Zhang, Radial symmetry and Liouville theorem for master equations, arXiv:2306.11554,
2023.
\bibitem[32]{34}  L. Ma, Z. Zhang, Monotonicity for fractional Laplacian systems in unbounded Lipschitz domains, Discrete Contin. Dyn. Syst., 41 (2021), 537-552.
\bibitem[33]{M1}L. Ma, Z. Zhang, Symmetry of positive solutions for Choquard equations with fractional p-Laplacian, Nonlinear Anal., 182 (2019), 248-262.

\bibitem[34]{M} J. Moser, On Harnack's theorem for elliptic differential equations, Comm. Pure Appl Math., 14 (1961), 577-591.
\bibitem[35]{37} M. Raberto, E. Scalas, and F. Mainardi, Waiting-times and returns in high-frequency fnancial data: an empirical study, Physica A, 314 (2002), 749-755.
\bibitem[36]{R}M. Riesz, Integrales de Riemann-Liouville et potentiels, Acta Sci. Math. Szeged, 9 (1938), 1-42.
\bibitem[37]{S} J. Serra, Regularity for fully nonlinear nonlocal parabolic equations with rough kernels, Calc. Var., 54 (2015),
615-629.
\bibitem[38]{SZ} J. Serrin, H. Zou, Cauchy-Liouville and universal boundedness theorems for quasilinear elliptic equations and
inequalities, Acta Math., 189 (2002),79-142.
\bibitem[39]{W}D. V. Widder, The heat equation, Academic Press, New York, 1976.

\bibitem[40]{WuC} L. Wu, W. Chen, Ancient solutions to nonlocal parabolic equations, Adv. Math., 408 (2022), 108607.






















\end{thebibliography}
\end{document}